\documentclass{amsart}
\usepackage{amsmath, amsfonts, amsthm, amssymb, mathrsfs, hyperref}

\newtheorem{theorem}{Theorem}
\newtheorem{lemma}{Lemma}
\newtheorem*{theorem*}{Theorem}
\theoremstyle{definition}

\newcommand{\N}{\mathbb{N}}
\newcommand{\R}{\mathbb{R}}
\newcommand{\Z}{\mathbb{Z}}
\newcommand{\y}{\mathbf{y}}
\newcommand{\x}{\mathbf{x}}
\newcommand{\q}{\mathbf{q}}
\newcommand{\p}{\mathbf{p}}
\newcommand{\s}{\mathbf{s}}
\newcommand{\aq}{\mathbf{a}_q}
\newcommand{\scrA}{\mathscr{A}}

\title[Inhomogeneous Khintchine theorem with polynomial decay]{Khintchine's theorem for inhomogeneous \makebox[\textwidth]{simultaneous approximation with polynomial decay}}
\author{Seongmin Kim}
\address{Department of Mathematical Sciences, Seoul National University, 1, Gwanak-ro, Gwanak-gu, Seoul, 08826, Republic of Korea}
\email{seongmin.kim@snu.ac.kr}
\date{}

\begin{document}

\begin{abstract}
    Khintchine's theorem on the measure dichotomy for the set of $\psi$-approximable numbers has been generalized to inhomogeneous and higher-dimensional settings. Allen and Ram\'irez conjectured that the monotonicity condition can be removed in the inhomogeneous $nm=2$ cases. In this paper, we resolve the $(n,m)=(1,2)$ case for $\psi$ satisfying a polynomial decay condition $\psi(q)=O(q^{-\delta})$ for some $\delta>0.$
\end{abstract}

\maketitle

\section{Introduction}
Khintchine's theorem \cite{Khi24} on metric Diophantine approximation states that the set of $\psi$-approximable numbers:
$$\scrA_{1,1}(\psi) := \{x\in[0,1): |qx-p|<\psi(q) \text{ for infinitely many } (q,p)\in\N\times\Z\}$$
has full or zero Lebesgue measure if the sum $\sum_{q=1}^\infty \psi(q)$ diverges or converges, respectively, for a monotonic function $\psi:\N \to \R_{\geq 0}.$ This theorem has been generalized to simultaneous approximation \cite{Khi26}, approximation of systems of linear forms \cite{Gro38}, inhomogeneous approximation \cite{Szu58}, and inhomogeneous approximation of systems of linear forms \cite{Sch64} which can be stated as follows. 

Let $n,m \in \N.$ For a function $\psi:\N\to \R_{\geq0}$ and a vector $\y \in \R^m,$ we say that an $n\times m$ matrix $\x \in \R^{nm}$ is \emph{$(\psi,\y)$-approximable} if there are infinitely many pairs of integer row vectors $(\q,\p) \in (\Z^n\setminus\{\mathbf{0}\}) \times \Z^m$ satisfying
    $$|\q\x - \p - \y | < \psi(|\q|),$$
where $|\cdot|$ denotes the maximum norm on $\R^m$ or $\R^n.$ We define $\scrA_{n,m}(\psi,\y)$ as the set of $(\psi,\y)$-approximable matrices in $[0,1)^{nm}.$ Then we can state the inhomogeneous, higher-dimensional generalization of Khintchine's theorem as follows.
\begin{theorem*}[Inhomogeneous Khintchine--Groshev]
    Let $n,m\in\N,\ \psi:\N \to \R_{\geq 0}$ and $\y \in \R^m.$ Then we have
    $$|\scrA_{n,m}(\psi,\y)| = \begin{cases} 0 & \text{ if \ } \sum_{q=1}^\infty q^{n-1}\psi(q)^m<\infty, \\ 1 & \text{ if \ } \sum_{q=1}^\infty q^{n-1}\psi(q)^m = \infty \text{ and } \psi \text{ is monotonic,}\end{cases}$$
    where $|\cdot|$ here denotes the Lebesgue measure of the set.
\end{theorem*}

The monotonicity condition in the divergence case cannot be dropped when $(n,m)=(1,1)$ by the counterexamples in \cite{DS41} for $y=0$ and \cite{Ram17} for any $y\in\R.$ In the homogeneous case $(\y=\mathbf0),$ Schmidt \cite{Sch60} removed the monotonicity condition for $n\geq3, $ Gallagher \cite{Gal65} removed it for $n=1$ and $m\geq2,$ and finally Beresnevich and Velani \cite{BV10} completed it for all cases $nm\geq2.$ 

In the inhomogeneous case $(\y \in \R^m),$ Sprind\v{z}uk's book \cite{Spr79} contains the proof of the inhomogeneous Khintchine--Groshev theorem without monotonicity for $n\geq3,$ followed by the proof by Yu \cite{Yu21} for $n=1,\ m\geq3$ and Allen--Ram\'irez \cite{AR23} for $nm\geq3.$ Allen and Ram\'irez conjectured that the monotonicity condition can also be removed for the inhomogeneous $nm=2$ cases. Hauke \cite{Hau23} proved this conjecture for the $(n,m)=(2,1)$ case with non-Liouville irrational $y\in\R,$ and the author \cite{Kim25} resolved the $(n,m)=(2,1)$ case with any $y\in\R.$

In this paper, we focus on the $(n,m)=(1,2)$ case of the Allen--Ram\'irez conjecture. In other words, we want to show that if $\sum_{q=1}^\infty \psi(q)^2=\infty,$ then $|\scrA_{1,2}(\psi,\y)|=1.$ Note that this is true when $\y=\mathbf0$ by \cite{BV10} and when $\y \in \mathbb{Q}^2$ by \cite{Kim25}. Also, Yu \cite{Yu21} showed that $|\scrA_{1,2}(\psi,\y)|=1$ if $\sum_{q=1}^\infty \psi(q)^2 (\frac{\varphi(q)}{q})^2=\infty,$ and Allen--Ram\'irez \cite{AR23} relaxed the condition to $\sum_{q=1}^\infty \psi(q)^2 (\frac{\varphi(q)}{q})^{1+\epsilon}=\infty$ for some $\epsilon>0.$ We prove the conjecture under a polynomial decay condition on $\psi.$ 

\begin{theorem}\label{thm:main}
    Let $\delta>0$ be arbitrary and $\psi:\N \to \R_{\geq 0}$ be a function satisfying $\psi(q) =O( q^{-\delta}).$ For any $\y \in \R^2$, if $\sum_{q=1}^\infty \psi(q)^2 = \infty,$ then $|\scrA_{1,2}(\psi,\y)|=1.$
\end{theorem}

This theorem can be seen as an Erd\H{o}s--Vaaler-type result. When Duffin and Schaeffer presented counterexamples for removing the monotonicity condition in the  $(n,m)=(1,1)$ case \cite{DS41}, they also conjectured that if $\sum_{q=1}^\infty \psi(q)\frac{\varphi(q)}q = \infty,$ then
$$\scrA'_{1,1}(\psi) = \{x\in[0,1): |qx-p|<\psi(q) \text{ for inf. many }(q,p)\in \N\times\Z \text{ s.t.}\gcd(q,p)=1\}$$
has full measure. In the long history of this conjecture, Erd\H{o}s \cite{Erd70} showed that if $\psi(q)=\frac1q$ or 0, then the conjecture holds, and Vaaler \cite{Vaa78} improved the condition to $\psi(q)=O(q^{-1}).$ Note that this conjecture has been proved by Koukoulopoulos and Maynard \cite{KM20}.

Chow, Hauke, Pollington and Ram\'irez \cite{CHPR25} constructed a function $\psi$ by restricting the support of any decreasing function $f$ satisfying $\lim_{q\to \infty} f(q)=0$ and $\sum_{q=1}^\infty f(q) = \infty$ so that $\sum_{q=1}^\infty \psi(q)=\infty$ but $|\scrA_{1,1}(\psi)|=0.$ Theorem \ref{thm:main} rules out the possibility of constructing an analogous counterexample in the $(n,m)=(1,2)$ case via support restriction, provided that the underlying function $f$ satisfies the polynomial decay condition $f(q)=O(q^{-\delta}).$

The proof begins by defining a reduced set $\tilde A_q$ by the condition $\gcd(q,b_q\p+\aq)=1$ where $(\frac{\aq}{b_q})_{q\in\N}$ is a sequence of rational vectors approximating $\y.$ Similar definitions can also be found in \cite{Sch64, CT24, BHV24, Kim25, TZ26}. The main idea of the proof is to choose an appropriate speed of approximation of $\y$ by $\frac{\aq}{b_q}$, which guarantees that $\tilde A_q$ and $\tilde A_r$ do not intersect for $\gcd(q,r)\gg q^{1-\epsilon}.$

\vspace{0.3cm}
\noindent \textbf{Notation.} We use Vinogradov notation where $f\ll g$ means $f\leq Cg$ for some uniform constant $C>0$, and $f\asymp g$ for $f\ll g$ and $g\ll f$. For an integer vector $\mathbf{k}=(k_1,\cdots, k_m) \in\Z^m\setminus \{\mathbf{0}\},$ we denote $\gcd(\mathbf{k})=\gcd(k_1,\cdots,k_m).$

\section{Proof of Theorem \ref{thm:main}}

For $\delta>0$ and $\y \in \R^2,$ define a sequence of pairs $(\aq,b_q) \in \Z^2 \times \N$ satisfying
$$|b_q\y-\aq|<q^{-\frac{\delta}{\delta+3}}, \quad 1\leq b_q \leq q^{\frac{2\delta}{\delta+3}}, \quad \gcd(\aq,b_q)=1.$$
The existence of such a sequence is guaranteed by Dirichlet's theorem. For each $q \in \N,$ we define
\begin{align*}
    A_q &:= \{\x \in [0,1)^2: \exists \p \in \Z^2 \text{ s.t. }|q\x - \p - \y|<\psi(q)\}, \\
    \tilde A_q &:= \{\x \in [0,1)^2: \exists \p \in \Z^2 \text{ s.t. }|q\x - \p - \y|<\psi(q) \text{ and } \gcd(q,b_q\p+\aq)=1\}.
\end{align*}
Observe that $\scrA_{1,2}(\psi,\y)=\limsup_{q\to\infty}A_q$ and $\tilde A_q \subseteq A_q.$ Our plan is to show $|\limsup_{q\to\infty} \tilde A_q|=1$ which implies that $|\scrA_{1,2}(\psi,\y)|=1$. The following lemma gives the measure of each set.

\begin{lemma}\label{lem:basic size}
    For $\psi(q)\leq \frac12,$
    $$|A_q| = 4\psi(q)^2,\quad |\tilde A_q| = 4\psi(q)^2 \prod_{p|q, p\nmid b_q} (1-p^{-2})$$
    where $p$ runs over primes. In particular, $|\tilde A_q| \asymp |A_q|.$
\end{lemma}
\begin{proof}
    The first equality holds because $A_q$ consists of $q^2$ disjoint boxes in $\R^2/\Z^2,$ each having an area of  $\left(\frac{2\psi(q)}q\right)^2$. The identity for $|\tilde A_q|$ follows directly from Lemma 1 of \cite{Kim25}. Finally, the asymptotic relation is a consequence of the bounds
    \[1 \geq \prod_{p|q, p\nmid b_q} (1-p^{-2}) >\prod_{p:\text{ prime}} (1-p^{-2}) = \frac{6}{\pi^2}. \qedhere \]
\end{proof}

The condition $\sum_{q=1}^\infty \psi(q)^2=\infty$ implies that $\sum_{q=1}^\infty |A_q| =\infty$ and $\sum_{q=1}^\infty |\tilde A_q|=\infty.$ Since the sequence of sets  $\{\tilde A _q\}_{q\in\N}$ lacks the pairwise independence required for the second Borel--Cantelli lemma, we instead use the divergence Borel--Cantelli lemma, which originates from the Chung--Erd\H{o}s inequality:
$$\left|\limsup_{q \to \infty} \tilde A_q\right| \geq \limsup_{Q \to \infty} \frac{\left(\sum_{q=1}^Q |\tilde A_q|\right)^2}{\sum_{q,r=1}^Q |\tilde A_q \cap \tilde A_r|}.$$
To establish that $|\limsup_{q\to\infty} \tilde A_q|=1,$ it suffices to  show that the right-hand side of the above inequality is strictly positive. The transition from positive measure to full measure is guaranteed by the  fact that $\tilde A_q$ is well-distributed: namely, $|\tilde A_q \cap U| \gg  |\tilde A_q||U|$ holds for any open set $U\subset [0,1)^2$ provided $q$ is sufficiently large.

The preceding argument summarizes the proof of the following lemma. For the complete proof, we refer to Theorem 4 of \cite{Kim25}, which refines methods such as those found in \cite{BDV06, AR23, AR25}, whose roots can be traced back to \cite{Kno26}.

\begin{lemma}
    Suppose that $\psi(q) \leq \frac12$ for every $q\in\N$ and $\sum_{q=1}^\infty \psi(q)^2 = \infty.$ If the sequence of sets $\{\tilde A_q\}_{q\in\N}$ is quasi-independent on average, that is,
    $$\limsup_{Q \to \infty} \frac{\left(\sum_{q=1}^Q |\tilde A_q|\right)^2}{\sum_{q,r=1}^Q |\tilde A_q \cap \tilde A_r|}>0,$$
    then $\left|\limsup_{q\to \infty}\tilde A_q \right|=1.$
\end{lemma}

Since $\tilde A_q \cap \tilde A_r \subseteq A_q \cap A_r,$ the standard overlap estimate for $|A_q \cap A_r|$ (see e.g. Lemma 8 of \cite{AR23}) trivially applies to $|\tilde A_q \cap \tilde A_r|,$ yielding the following bound.
\begin{lemma}\label{lem:overlap}
    For $\psi(q),\psi(r) \leq \frac12,$ we have
    $$|\tilde A_q \cap \tilde A_r| \ll \psi(q)^2\psi(r)^2 + \psi(q)^2 \frac{\gcd(q,r)^2}{q^2}.$$
\end{lemma}

This overlap estimate is not sufficient to show the quasi-independence on average, because
$$\sum_{r=1}^q \frac{\gcd(q,r)^2}{q^2} =\sum_{d|q} \sum_{\substack{1\leq r' \leq q/d \\ \gcd(r',q/d)=1}} \frac{d^2}{q^2} = \sum_{d|q} \frac{d^2}{q^2}\varphi(q/d) = \sum_{e|q} \frac{\varphi(e)}{e^2}$$
is unbounded as $q \to \infty.$ Therefore, we present the following key lemma which excludes the case when $\gcd(q,r)$ is relatively big.

\begin{lemma}\label{lem:key}
    If $q>r,\ \psi(q)\leq q^{-\delta},\ \psi(r)\leq r^{-\delta}$ and $\gcd(q,r) > 4q^{\frac{3}{\delta+3}},$ then $|\tilde A_q \cap \tilde A_r|=0.$
\end{lemma}
\begin{proof}
    We will prove that $\tilde A_q \cap A_r = \emptyset,$ which is a stronger result. Let $U(\epsilon)$ be a box $(-\epsilon, \epsilon)^2 \subset \R^2$ centered at the origin. Observe that $\tilde A_q$ and $A_r$ can be expressed as unions of small boxes:
    \begin{align*}
        \tilde A_q&=\bigcup_{\substack{\p \in \Z^2 \\ \gcd(q,b_q\p+\aq)=1}} \left( \frac{\p+\y}{q} + U\left(\frac{\psi(q)}{q}\right) \right) \cap [0,1)^2, \\
        A_r &= \bigcup_{\s \in \Z^2} \left( \frac{\s+\y}{r} + U\left(\frac{\psi(r)}r\right)\right) \cap [0,1)^2.
    \end{align*}
    The distance between the centers of the boxes is
    \begin{align*}
        \left|\frac{\p+\y}q - \frac{\s+\y}r\right| &= \left|\frac{\p+\frac{\aq}{b_q} +\y -\frac{\aq}{b_q}}q - \frac{\s+\frac{\aq}{b_q}+\y-\frac{\aq}{b_q}}r\right| \\
        &= \left|\frac{1}{b_q}\left(\frac{b_q\p+\aq}{q} - \frac{b_q\s + \aq}r\right) + \left(\y-\frac{\aq}{b_q}\right)\left(\frac1q - \frac1r\right)\right|.
    \end{align*}
    Since $\gcd(q,b_q\p+\aq)=1,$ the rational vector $\frac{b_q\p+\aq}{q}$ cannot be reduced to have smaller denominator, and thus it is different from $\frac{b_q\s + \aq}r.$ Hence, we have
    $$\left|\frac{1}{b_q}\left(\frac{b_q\p+\aq}{q} - \frac{b_q\s + \aq}r\right)\right| \geq \frac1{b_q\mathrm{lcm}(q,r)}=\frac{\gcd(q,r)}{b_qqr}.$$
    From the definition of $(\aq, b_q)$ and $\gcd(q,r) > 2q^{\frac{3}{\delta+3}},$
    $$\left|\left(\y-\frac{\aq}{b_q}\right)\left(\frac1q - \frac1r\right)\right| < \frac{1}{b_qq^{\frac{\delta}{\delta+3}}} \cdot \frac1r=\frac{q^{\frac{3}{\delta+3}}}{b_qqr}<\frac{\gcd(q,r)}{2b_qqr}.$$
    Combining the above inequalities gives
    $$\left|\frac{\p+\y}q - \frac{\s+\y}r\right| > \frac{\gcd(q,r)}{2b_qqr}.$$
    On the other hand, since $\psi(q)\leq q^{-\delta} < r^{-\delta},\ \psi(r)\leq r^{-\delta} \leq \gcd(q,r)^{-\delta}<q^{-\frac{3\delta}{\delta+3}},$ $\gcd(q,r)>4q^{\frac{3}{\delta+3}}$ and $b_q\leq q^{\frac{2\delta}{\delta+3}},$ we have
    $$\frac{\psi(q)}q + \frac{\psi(r)}r < \frac{2}{r\cdot r^\delta} < \frac2{r \cdot q^{\frac{3\delta}{\delta+3}}} = \frac{2q^{\frac{3}{\delta+3}}}{q^{\frac{2\delta}{\delta+3}} \cdot qr} \leq \frac{2q^{\frac{3}{\delta+3}}}{b_qqr} < \frac{\gcd(q,r)}{2b_qqr}.$$
    Therefore,
    $$\frac{\psi(q)}q + \frac{\psi(r)}r < \frac{\gcd(q,r)}{2b_qqr} < \left|\frac{\p+\y}q - \frac{\s+\y}r\right|$$
    and it follows that
    $$\left( \frac{\p+\y}{q} + U\left(\frac{\psi(q)}{q}\right) \right) \cap \left( \frac{\s+\y}{r} + U\left(\frac{\psi(r)}r\right)\right) = \emptyset$$
    for every $\p,\s \in \Z^2$ satisfying $\gcd(q, b_q\p+\aq)=1.$ Consequently, $\tilde A_q \cap A_r = \emptyset.$
\end{proof}

\begin{proof}[Proof of Theorem \ref{thm:main}] We may assume that $\psi(q) \leq \frac12$ and $\psi(q) \leq q^{-\delta}$ for every $q \in \N$ since if we let $\psi'(q) :=\min\{c\psi(q), \frac12\}$ for small enough $0<c\leq1,$ then $\sum_{q=1}^\infty \psi(q)^2 = \infty$ implies $\sum_{q=1}^\infty \psi'(q)^2= \infty$ and $\scrA_{1,2}(\psi',\y) \subseteq \scrA_{1,2}(\psi,\y).$

We now show the quasi-independence of the sequence $\{\tilde A_q\}_{q\in \N}.$ Since $|\tilde A_q \cap \tilde A_r| = 0$ for $q>r$ and $\gcd(q,r)>4q^{\frac3{\delta+3}}$ by Lemma \ref{lem:key}, the estimates from Lemmas \ref{lem:basic size} and \ref{lem:overlap} give
    \begin{align*}
        \sum_{r=1}^q |\tilde A_q \cap \tilde A_r| &= |\tilde A_q| + \sum_{\substack{1\leq r < q \\ \gcd(q,r) \leq 4q^{\frac{3}{\delta+3}}}} |\tilde A_q \cap \tilde A_r| \\
        &\ll \sum_{\substack{1\leq r \leq q \\ \gcd(q,r) \leq 4q^{\frac{3}{\delta+3}}}} \left( \frac{\psi(q)^2}{q^2} \gcd(q,r)^2 + \psi(q)^2\psi(r)^2 \right).
    \end{align*}
    Here, the sum of the first term is bounded by
    \begin{align*}
        \sum_{\substack{1\leq r \leq q \\ \gcd(q,r) \leq 4q^{\frac{3}{\delta+3}}}} \frac{\gcd(q,r)^2}{q^2} = \sum_{\substack{d|q \\ d\leq4q^{\frac{3}{\delta+3}}}} \frac{d^2}{q^2}\varphi(q/d) = \sum_{\substack{e|q \\ e\geq\frac14q^{\frac{\delta}{\delta+3}}}} \frac{\varphi(e)}{e^2} \leq \sum_{\substack{e|q \\ e\geq\frac14q^{\frac{\delta}{\delta+3}}}} \frac1e < \frac{4\tau(q)}{q^{\frac{\delta}{\delta+3}}},
    \end{align*}
    where $\tau(q)$ is the number of divisors of $q$. Since $\tau(q) \ll q^{\frac{\delta}{\delta+3}}$ (see for example, \cite[Theorem 315]{HW08}), we have
    $$\sum_{r=1}^q |\tilde A_q \cap \tilde A_r| \ll \psi(q)^2 + \sum_{r=1}^q \psi(q)^2\psi(r)^2 \asymp |\tilde A_q| + \sum_{r=1}^q |\tilde A_q||\tilde A_r|.$$
    The condition $\sum_{q=1}^\infty \psi(q)^2 =\infty$ implies 
    $$\sum_{q,r=1}^Q |\tilde A_q \cap \tilde A_r| \ll \sum_{q=1}^Q \sum_{r=1}^q |\tilde A_q \cap \tilde A_r| \ll \sum_{q=1}^Q \left(|\tilde A_q| + \sum_{r=1}^q |\tilde A_q||\tilde A_r|\right) \ll \left(\sum_{q=1}^Q |\tilde A_q|\right)^2.$$
    Therefore, $|\limsup_{q\to\infty}\tilde A_q|=1,$ and thus $|\scrA_{1,2}(\psi,\y)|=1.$
\end{proof}

\noindent \textbf{Acknowledgement. } The author would like to thank Seonhee Lim for her guidance and helpful discussion. The author is supported by National Research Foundation of Korea, under project number RS-2025-00515082 and RS-2025-02293115.

\end{document}